\documentclass[a4paper,12pt]{article}

\usepackage[latin1]{inputenc}
\usepackage[swedish,english]{babel}
\usepackage{amsmath,amsthm}
\usepackage{amssymb}
\usepackage{graphicx}
\usepackage{epstopdf}
\usepackage{amsmath}
\usepackage{amssymb}
\usepackage{subfig}
\usepackage{booktabs}   
\usepackage{multirow}

\theoremstyle{plain}
\newtheorem{theorem}{Theorem}
\newtheorem{lemma}{Lemma}

\theoremstyle{definition}

\newtheorem{definition}{Definition}

\theoremstyle{remark}

\newcommand{\vertiii}[1]{{\left\vert\kern-0.25ex\left\vert\kern-0.25ex\left\vert #1 
    \right\vert\kern-0.25ex\right\vert\kern-0.25ex\right\vert}}
\newcommand{\bs}{\boldsymbol}

\title{An improved high order finite difference method for non--conforming grid interfaces for the wave equation} 

\author{Siyang Wang\thanks{Department of Information Technology, Uppsala University, 752 37 Uppsala, Sweden. \emph{Current address:} Department of Mathematical Sciences, Chalmers University of Technology and University of Gothenburg,
412 96 Gothenburg, Sweden. Email: siyang.wang@chalmers.se}}

\date{}

\begin{document}

\selectlanguage{english}

\maketitle

\begin{abstract} 
This paper presents an extension of a recently developed high order finite difference method for the wave equation on a grid with non--conforming interfaces. The stability proof of the existing methods relies on the interpolation operators being \emph{norm--contracting}, which is satisfied by the second and fourth order operators,  but not by the sixth order operator. We construct new penalty terms to impose interface conditions such that the stability proof does not require the \emph{norm--contracting} condition. As a consequence, the sixth order accurate scheme is also provably stable. Numerical experiments demonstrate the improved stability and accuracy property. 

\end{abstract}

\section{Introduction}
Wave propagation can be modeled by hyperbolic partial differential equations (PDEs). When solving a hyperbolic PDE by a finite difference method, to achieve a certain accuracy a minimum number of grid points per wavelength is required. This number is smaller with a high order method than with a low order method, which makes high order finite difference methods more efficient  to solve wave propagation problems on smooth domains, see the pioneering paper \cite{Kreiss1972} for first order hyperbolic PDEs, and the recent work \cite{Hagstrom2012} for second order hyperbolic PDEs.

On a uniform grid, high order central finite difference stencils are easily constructed by Taylor expansions \cite{Fornberg1998}. Close to boundaries, one--sided stencils can be used. It is important that the boundary closure is accurate, and the numerical scheme is stable. One successful approach is to use a finite difference operator satisfying a summation--by--parts (SBP) property \cite{Kreiss1974}.  When an SBP operator acts on a grid function, it mimics the integration--by--parts principle in the continuous setting. An energy estimate can be obtained if boundary conditions are imposed appropriately, for example by using the simultaneous--approximation--term (SAT) method \cite{Carpenter1994} or ghost points \cite{Petersson2010}. A scheme satisfying an energy estimate is called energy stable \cite{Gustafsson2008,Gustafsson2013}.

Finite difference methods in the basic form can only be used on box--shaped domains. When complex geometry is present, the domain can be partitioned into blocks to resolve the geometrical feature. Each block has four sides and is mapped to a reference domain. If the corners of adjacent blocks meet, we say they are conforming blocks; otherwise they are non--conforming. In addition, a grid interface is conforming if no hanging nodes are present. When partitioning a domain, we can always make the blocks and interfaces conforming. However, in many situations it is desirable to use a more flexible strategy of partition that leads to non--conforming blocks and grid interfaces.   

As an example, we consider a wave traveling in a heterogeneous medium with the wave speed varying in space. The wavelength is proportional to the wave speed for a given frequency. For accuracy the grid spacing is determined by the shortest wavelength. If a uniform grid is used in the entire computational domain, then the grid spacing must be small enough to resolve the shortest wavelength, resulting in an unnecessarily fine grid elsewhere. It is then more efficient to construct a grid according to the wavelength in each block, which leads to non--conforming interfaces with hanging nodes. 

If only conforming blocks are used, the domain partitioning may end up with many blocks of small size. To use a high order finite difference method, a minimum number of grid points is required in each block due to the stencil width. This then results in unnecessarily many grid points in the small blocks, and consequently a suboptimal performance of the numerical scheme. In such a situation, non--conforming blocks are more appropriate.   

In an SBP finite difference method, interface conditions can also be imposed by the SAT method \cite{Carpenter1999,Carpenter2010} or ghost points \cite{Petersson2010}. In the SBP--SAT framework, wave propagation in a heterogeneous medium with complex geometry is considered in \cite{Virta2014}. A stable and accurate multi--block finite difference method with conforming grid interfaces and blocks is presented. The focus in \cite{Wang2016} is the numerical treatment of non--conforming interfaces and blocks by using SBP--preserving interpolation operators \cite{Kozdon2016,Mattsson2010}. Energy stability is proved with an assumption that the interpolation  operators are \emph{norm--contracting}. In the same paper \cite{Wang2016} , it is verified that not all interpolation operators satisfy this assumption, and instability occurs when the sixth order method is used on a domain with non--conforming, curved interfaces. 

In this paper, we construct new penalty terms in the SBP--SAT finite difference framework for the numerical interface treatment. The resulting scheme is energy stable even when the interpolation operators are not  \emph{norm--contracting}. This extends the provably stable scheme from fourth order accuracy \cite{Wang2016} to sixth order accuracy. The technique can be potentially used to construct even higher order schemes, provided that the corresponding SBP operators exist. Another contribution of this paper is the numerical treatment of non--conforming blocks and interfaces on curvilinear grids, where as in \cite{Wang2016} such a case is studied on Cartesian grids. We  also conduct numerical experiments to verify that the new sixth order scheme is stable with non--conforming, curved interfaces. 

The paper is organized as follows. In Section 2, we introduce the SBP--SAT finite difference method. In Section 3, we consider the wave equation on a Cartesian grid and present the new penalty terms for numerical interface treatments. Stability is proved by the energy method. We then generalize the scheme to non--conforming blocks and grid interfaces on curvilinear grids in Section 4. Numerical experiments are presented in Section 5 to verify the stability and accuracy property of the developed scheme. We draw conclusion in Section 6.

\section{SBP--SAT finite difference methods}
Finite difference operators satisfying an SBP property have been widely used to discretize time dependent PDEs. An SBP operator has central finite difference stencils in the interior, and special one--sided stencils at a few grid points near boundaries. The boundary stencils are chosen so that the operator satisfies a summation--by--parts property, which is the discrete counterpart of the integration--by--parts principle. With the SAT method imposing boundary and interface conditions, the SBP--SAT finite difference method possesses a great advantage: it is possible to prove energy stability for high order accurate schemes for  initial--boundary--value problems. 

To introduce the SBP--SAT finite difference method, we consider the one dimensional domain $[0,1]$ discretized by the grid points $x_j=jh, j=0,1,\cdots,N$ with a constant grid spacing $h=1/N$. We use the capital letter, for example, $U$, to denote a smooth function in $[0,1]$, and the corresponding small letter, $u$, to denote its values on the grid $u=[U(x_0),U(x_1),\cdots,U(x_N)]^T$. 

\subsection{Definitions of SBP operators}
The SBP concept and the first derivative SBP operator $D_1\approx\partial/\partial x$ are introduced in \cite{Kreiss1974}, and later refined in \cite{Strand1994}. Formally it is defined as follows.
\begin{definition}
A difference operator $D_1=H^{-1}Q$ approximating $\partial/\partial x$ is a diagonal norm first derivative SBP operator if $H$ is diagonal positive definite and $Q+Q^T=\text{diag}(-1,0,\cdots,0,1)$.
\end{definition}
The operator $H$ defines the SBP norm, and leads to the identity
\begin{equation}\label{D1_eq}
u^THD_1v=-(D_1u)^THv-u_0v_0+u_Nv_N,
\end{equation}
which is the discrete analogue of the integration--by--parts formula
\begin{equation*}
\int_0^1 UV_xdx=-\int_0^1 U_xVdx-U(0)V(0)+U(1)V(1),
\end{equation*}
since the norm $H$ is also a quadrature \cite{Fernandez2014b,Hicken2013}.

For the second derivative, we distinguish between a constant coefficient operator $D_2\approx\partial^2/\partial x^2$ and a variable coefficient operator $D_2^{(b)}\approx\partial/\partial x (b(x)\partial/\partial x)$ with a known function $b(x)>0$.
\begin{definition}\label{D2}
A difference operator $D_2=H^{-1}(-M+BS)$ approximating $\partial^2/\partial x^2$ is a diagonal norm second derivative SBP operator if $H$ is diagonal positive definite, $M$ is symmetric positive semi--definite, $B=\text{diag}(-1,0,\cdots,0,1)$, and the first and last row of $S$ approximate $\partial/\partial x$ at the two boundaries, respectively.
\end{definition}
Such an operator is constructed in \cite{Mattsson2004}. It is later found in  \cite{Appelo2007,Mattsson2008} that the operator $M$ in $D_2$ satisfies the following property.\begin{lemma}\label{lemmaM}
The symmetric positive semi--definite operator $M$ can be written as 
\begin{equation*}
M=\widetilde M + h\theta (BS)^TBS,
\end{equation*}
where $\widetilde M$ is also symmetric positive semi--definite, $\theta>0$ is a constant independent of $h$, $B$ and $S$ are the same as in Definition \ref{D2}.
\end{lemma}
Lemma \ref{lemmaM} is often referred to as the \emph{borrowing trick}, as we can \emph{borrow} from the positive semi--definite operator $M$ a small, mesh dependent amount, with the resulting operator $\widetilde M$ still positive semi--definite. This property is essential for energy stability of problems with interfaces or Dirichlet boundary conditions.

For the variable coefficient case we have correspondingly
\begin{definition}\label{D2b}
A difference operator $D_2^{(b)}=H^{-1}(-M^{(b)}+B^{(b)}S)$ approximating $\partial/\partial x (b(x)\partial/\partial x)$ is a diagonal norm second derivative variable coefficient SBP operator if $H$ is diagonal positive definite, $M^{(b)}$ is symmetric positive semi--definite, $B^{(b)}=\text{diag}(-b(x_0),0,\cdots,0,b(x_N))$, and the first and last row of $S$ approximate $\partial/\partial x$ at the two boundaries, respectively.
\end{definition}
Such an operator is constructed in \cite{Mattsson2010}, and the operator $M^{(b)}$ has the following two important properties \cite{Virta2014}.

\begin{lemma}\label{lemmaMb}
The symmetric positive semi--definite operator $M^{(b)}$ can be written as 
\begin{equation*}
M^{(b)}=\widetilde M^{(b)} + h\sigma b_m(BS)^TBS,
\end{equation*}
where $\widetilde M$ is also symmetric positive semi--definite, $\sigma>0$ is a constant independent of $h$, $B$ and $S$ are the same as in Definition \ref{D2}, and 
\begin{equation*}
b_m=\min (b(x_0), b(x_1),\cdots, b(x_l), b(x_N), b(x_{N-1}),\cdots,b(x_{N-l}))
\end{equation*}
with a constant $l$ independent of $h$.
\end{lemma}
Lemma \ref{lemmaMb} for the variable coefficient SBP operators is an analogue of Lemma \ref{lemmaM} for the constant coefficient case.  We note that $b_m$ is the smallest value of the variable coefficient $b(x)$ on the first and last $l$ grid points. The smaller $b_m$ is, the less we can \emph{borrow} from $M^{(b)}$.

\begin{lemma}\label{lemmaMc}
The SBP operator $D_2^{(b)}$ is compatible with $D_1$ if $M^{(b)}$ can be written as 
\begin{equation*}
M^{(b)}=D_1^T B^{(b)} H D_1 + R^{(b)},
\end{equation*}
where $R^{(b)}$ is symmetric positive semi--definite, and $B^{(b)}$ is the same as in Definition \ref{D2b}. 
\end{lemma}
Lemma \ref{lemmaMc} is essential for energy stability when mixed derivatives are present in the equation, for example the wave equation on curvilinear grids and the elastic wave equation.

The definitions and precise forms of the above operators can be found in \cite{Kreiss1974,Mattsson2012,Mattsson2010,Mattsson2004,Strand1994}. These operators have the minimal interior stencil width. In addition, they have the same associated norm $H$ for a given accuracy order. In the stability analysis, we only consider numerical treatment of interface conditions. As a consequence, the operator $B$  in the preceding lemmas only has one nonzero element, corresponding to the terms on the interface. 

The interior stencil of an SBP operator is the standard central finite difference stencil with truncation error $\mathcal{O}(h^{2p})$. On a few grid points near boundaries, special one-sided stencils are used to fulfill the SBP requirement with a larger truncation error $\mathcal{O}(h^p)$. Operators $D_1$ and $D_2$ with $p=1,2,3,4$ are constructed in \cite{Kreiss1974,Strand1994} and \cite{Mattsson2004}, respectively. The variable coefficient operators $D_2^{(b)}$ with $p=1,2,3$ are constructed in \cite{Mattsson2012}.

In this paper, we call the above SBP operators {\em $2p^{th}$ order accurate}. When using in a numerical scheme, we also call the scheme {\em $2p^{th}$ order accurate}, even though the truncation error of the numerical scheme may not be $\mathcal{O}(h^{2p})$ or $\mathcal{O}(h^p)$. In the discussion of accuracy, we make the truncation error of the scheme precise. 

\subsection{The SAT method}
An SBP operator only approximates a certain derivative, but does not impose any boundary condition. The boundary conditions must be imposed carefully so that an energy estimate can be obtained to ensure stability. This can be done by for example the SAT method \cite{Carpenter1994}, the projection method \cite{Mattsson2006,Olsson1995a,Olsson1995b} and the ghost points method \cite{Petersson2010}. In this paper, we choose the SAT method to impose both boundary and interface conditions, since in many cases it is easy to derive an energy estimate. The key ingredient of the SAT method is to add penalty terms to the semi--discretized equation and choose penalty parameters so that an energy estimate is obtained. This technique bears a similarity to the Nitsche's finite element method \cite{Sticko2016}, and the discontinuous Galerkin method \cite{Appelo2015,Grote2006}. Detailed discussions of the SBP--SAT finite difference methods can be found in \cite{Fernandez2014,Svard2014}. 

\section{The wave equation on a Cartesian grid}
We start by considering the wave equation in two space dimensions in a composite domain $\Omega = [0,1]^2$ with an interface $\Gamma$ at $x=0.5$. The left and right domain are denoted by $\Omega_u$ and $\Omega_v$, respectively, and the equations  are
\begin{align}\label{eqn_Cartesian}
&U_{tt}=U_{xx}+U_{yy},\ (x,y)\in\Omega_u, \\
&V_{tt}=V_{xx}+V_{yy},\ (x,y)\in\Omega_v.
\end{align}
At the interface the physical conditions are
\begin{equation}\label{physical_conditions_Cartesian}
U(0.5,y,t)=V(0.5,y,t),\ U_x(0.5,y,t)=V_x(0.5,y,t).
\end{equation}
For a wellposed problem, suitable boundary conditions must be imposed at the boundaries. As the focus in this paper is the numerical treatment of interface coupling, we exclude discussions on boundary conditions and the corresponding numerical techniques. We refer to \cite{Kreiss2012} for physical boundary conditions, and \cite{Mattsson2009} for the numerical techniques. 

To solve  \eqref{eqn_Cartesian}--\eqref{physical_conditions_Cartesian}, we start by generating a Cartesian grid in each domain independently with $n_{ux}\times n_{uy}$ grid points in $\Omega_u$ and $n_{vx}\times n_{vy}$ grid points in $\Omega_v$. We are particularly interested in a non--conforming interface when $n_{uy}\neq n_{vy}$. In this case, the solutions on the interface must be interpolated. We denote $I_{u2v}$ and $I_{v2u}$ interpolation operators that interpolate the solution from $\Omega_u$ to $\Omega_v$, and from $\Omega_v$ to $\Omega_u$, respectively. In the SBP--SAT finite difference framework, these operators must satisfy certain conditions so that the scheme could be energy stable. 

\begin{definition}
Let $H_{u}$ and $H_{v}$ denote the SBP norms on the interface for the grid in $\Omega_u$ and $\Omega_v$, respectively. The interpolation operators $I_{u2v}$ and $I_{v2u}$ are \emph{norm--compatible} if 
\begin{equation}\label{nc}
H_{u}I_{v2u}=(H_{v}I_{u2v})^T.
\end{equation}
\end{definition}
In \cite{Wang2016}, it is also defined that the interpolation operators are \emph{norm--contracting} if the two operators
\begin{equation*}
H_{u}(I_{u}-I_{v2u}I_{u2v})\text{ and } H_{v}(I_{v}-I_{u2v}I_{v2u})
\end{equation*}
are symmetric positive semi--definite, where $I_u$ and $I_v$ are identity operators.

Norm--compatible interpolation operators are first constructed in \cite{Mattsson2010} for the case of a 1:2 mesh refinement ratio, and are extended to an arbitrary ratio in \cite{Kozdon2016}. The accuracy property of these interpolation operators has a similar fashion as the corresponding SBP operators. More precisely, the interpolation error is $\mathcal{O}(h^{2p})$ in the interior of the interface, and $\mathcal{O}(h^{p})$ on a few grid points near the edge of the interface. Therefore, the interpolation is exact only for polynomials of order up to $p-1$. In \cite{Lundquist}, it is proved that it is not possible to construct norm--compatible interpolation operators $I_{u2v}$ and $I_{v2u}$ such that both  interpolate polynomials of order $p$ or higher. 

A stable SBP--SAT finite difference method for solving \eqref{eqn_Cartesian}--\eqref{physical_conditions_Cartesian} is presented in \cite{Wang2016}. Energy stability is proved by assuming the interpolation operators are \emph{norm--compatible} and \emph{norm--contracting}. While the  \emph{norm--compatible} condition can be constrained when constructing the operators, it is not easy to take into account the \emph{norm--contracting} condition.  In fact, the interpolation operators with higher than fourth order accuracy in \cite{Kozdon2016,Mattsson2010} are not \emph{norm--contracting}. 

Below we present a new way of imposing the interface conditions \eqref{physical_conditions_Cartesian} with the advantage that an energy estimate is obtained without requiring the interpolation operators to be \emph{norm--contracting}. For cleaner notations, terms imposing boundary conditions are omitted.

Equation \eqref{eqn_Cartesian}--\eqref{physical_conditions_Cartesian} are discretized in space as
\begin{align}
&u_{tt}=\bs{D_u}u+SAT_{u1}+SAT_{u2}+SAT_{u3}+SAT_{\partial u},\label{semi1_Cartesian} \\
&v_{tt}=\bs{D_v}v+SAT_{v1}+SAT_{v2}+SAT_{v3}+SAT_{\partial v},\label{semi2_Cartesian}
\end{align}
where
\begin{subequations}
\begin{align}
&SAT_{u1} = \frac{1}{2} \bs{H_{ux}^{-1}} \bs{S_{ux}^T}(\bs{E_{ux}}u - (E_{uv}\otimes I_{v2u})v),  \label{SATu1} \\
&SAT_{u2} = -\frac{\tau}{2} \bs{H_{ux}^{-1}} (\bs{E_{ux}}u - (E_{uv}\otimes I_{v2u})v) \label{SATu2}, \\
&SAT_{u3} = -\frac{\tau}{2} \bs{H_{ux}^{-1}} ( (E_{ux}\otimes(I_{v2u}I_{u2v}) )u - (E_{uv}\otimes I_{v2u})v), \label{SATu3}\\
&SAT_{\partial u} = -\frac{1}{2} \bs{H_{ux}^{-1}} (\bs{E_{ux}}\bs{S_{ux}}u-(E_{uv}\otimes I_{v2u})\bs{S_{vx}}v), \label{SATu4} 
\end{align}
\end{subequations}
and
\begin{subequations}
\begin{align}
&SAT_{v1} = -\frac{1}{2} \bs{H_{vx}^{-1}} \bs{S_{vx}^T}(\bs{E_{vx}}v - (E_{vu}\otimes I_{u2v})u ),\\
&SAT_{v2} = -\frac{\tau}{2} \bs{H_{vx}^{-1}} (\bs{E_{vx}}v - (E_{vu}\otimes I_{u2v})u), \\
&SAT_{v3} = -\frac{\tau}{2} \bs{H_{vx}^{-1}} ( (E_{vx}\otimes(I_{u2v}I_{v2u}) )v - (E_{vu}\otimes I_{u2v})u),  \label{SATv3} \\
&SAT_{\partial v} = \frac{1}{2} \bs{H_{vx}^{-1}} (\bs{E_{vx}}\bs{S_{vx}}v-(E_{vu}\otimes I_{u2v})\bs{S_{ux}}u). 
\end{align}
\end{subequations}
The numerical solution vectors $u$ and $v$ approximate the true solution $U$ and $V$, respectively. The solution vectors are arranged column--wise, i.e. the first few elements of $u$ and $v$ correspond to the solutions on the left boundary of $\Omega_u$ and $\Omega_v$, respectively. Most operators in two space dimensions can be extended from the corresponding one dimensional operators by using a Kronecker product $\otimes$. Such two dimensional operators are denoted by bold letters, with the subscript indicating the spatial direction and the grid function that the operator is associated to. For example, the operator $\bs{H_{ux}^{-1}}$ equals to $H_{ux}^{-1}\otimes I_{uy}$, where $H_{ux}^{-1}$ is the inverse of the SBP norm in the $x$--direction acting on $u$, and $I_{uy}$ is an identity operator. The operator $\bs{E}$ extracts the numerical solution at the interface. $\bs{D_u}u$ and $\bs{D_v}v$ are SBP approximations of $U_{xx} + U_{yy}$ and $V_{xx} + V_{yy}$, respectively.

We compare the above scheme with the ones described in \cite{Virta2014,Wang2016} by discussing the penalty terms \eqref{SATu1}--\eqref{SATu4}. A term like $\bs{E_{ux}}u$ used in \cite{Virta2014,Wang2016} is broken into two parts: $1/2\boldsymbol{E_{ux}}u$ in \eqref{SATu2} and $1/2E_{ux}\otimes I_{v2u}I_{u2v}u$ in \eqref{SATu3}.  Note the relation between them:  $E_{ux}\otimes I_{v2u}I_{u2v}u$ is just $\boldsymbol{E_{ux}}u$ interpolated to the grid on the interface of $\Omega_v$, then interpolated back to the grid of $\Omega_u$. Since the interpolation is not exact, $E_{ux}\otimes I_{v2u}I_{u2v}u$ differs from $\boldsymbol{E_{ux}}u$ by the truncation error of the interpolation operators. It is this change in penalty terms that makes the scheme stable without requiring the interpolation operators to be \emph{norm--contracting}. We summarize the stability result in the following theorem.

\begin{theorem}\label{thm1}
With norm--compatible interpolation operators, the semi--discretization \eqref{semi1_Cartesian}--\eqref{semi2_Cartesian} is stable for any $\tau$ such that 
\begin{equation}\label{tau_limit}
\tau\geq\max\left( \frac{1}{2\theta h_{ux}}, \frac{1}{2\theta h_{vx}} \right),
\end{equation}
where $\theta$ is the constant in Lemma \ref{lemmaM}, and $h_{ux}$ and $h_{vx}$ are the mesh size in the $x$--direction in $\Omega_u$ and $\Omega_v$, respectively. 
\end{theorem}

\begin{proof}
We prove stability by the energy method. Multiplying from the left of \eqref{semi1_Cartesian} by $u_t^T (H_{ux}\otimes H_{uy})$ and  \eqref{semi2_Cartesian} by $v_t^T (H_{vx}\otimes H_{vy})$, we obtain
\begin{align*}
&u_t^T (H_{ux}\otimes H_{uy}) u_{tt} + v_t^T (H_{vx}\otimes H_{vy}) v_{tt}\\
=\ & u_t^T (H_{ux}\otimes H_{uy}) \bs{D_u}u + v_t^T (H_{vx}\otimes H_{vy}) \bs{D_v}v\\
&+ \frac{1}{2}  u_t^T \bs{H_{uy}} \bs{S_{ux}^T}(\bs{E_{ux}}u - (E_{uv}\otimes I_{v2u})v) \\
&-\frac{\tau}{2} u_t^T \bs{H_{uy}} (\bs{E_{ux}}u - (E_{uv}\otimes I_{v2u})v) \\
&-\frac{\tau}{2} u_t^T\bs{H_{uy}} ( (E_{ux}\otimes(I_{v2u}I_{u2v}) )u - (E_{uv}\otimes I_{v2u})v) \\
&-\frac{1}{2}  u_t^T \bs{H_{uy}} (\bs{E_{ux}}\bs{S_{ux}}u-(E_{uv}\otimes I_{v2u})\bs{S_{vx}}v)\\
&-\frac{1}{2} v_t^T \bs{H_{vy}} \bs{S_{vx}^T}(\bs{E_{vx}}v - (E_{vu}\otimes I_{u2v})u) \\
&-\frac{\tau}{2}  v_t^T\bs{H_{vy}} (\bs{E_{vx}}v - (E_{vu}\otimes I_{u2v})u) \\
&-\frac{\tau}{2}  v_t^T\bs{H_{vy}} ( (E_{vx}\otimes(I_{u2v}I_{v2u}) )v - (E_{vu}\otimes I_{u2v})u) \\
& +\frac{1}{2} v_t^T \bs{H_{vy}} (\bs{E_{vx}}\bs{S_{vx}}v-(E_{vu}\otimes I_{u2v})\bs{S_{ux}}u).
\end{align*}

We note that the left--hand side of the above equation can be written as the time derivative of a quadratic term. The main idea of deriving an energy estimate is to move all terms on the right--hand side to the left, and determine the penalty parameter $\tau$ so that all terms on the left--hand side  can be written as the time derivative of a non--negative quantity, i.e. the discrete energy. 

To do so, we use the \emph{borrowing trick} in Lemma \ref{lemmaM} for the SBP operators in the $x$--direction, and the \emph{norm--compatible} property \eqref{nc} of the interpolation operators, to obtain the change of energy $\mathbf{G}$ as
 \begin{align}\label{ddtG}
\frac{d}{dt} \mathbf{G}=\frac{d}{dt} (\mathbf{G_1}+\mathbf{G_2}+\mathbf{G_3})=0,
\end{align}
where 
\begin{align*}
 \mathbf{G_1} =& u_t^T (H_{ux}\otimes H_{uy}) u_t + v_t^T (H_{vx}\otimes H_{vy}) v_t  \\
 &+ u^T (H_{ux}\otimes M_{uy}) u + v^T (H_{vx}\otimes M_{vy})v \\
 &+ u^T (\widetilde M_{ux}\otimes H_{uy}) u + v^T (\widetilde M_{vx}\otimes H_{vy}) v,\\
 \mathbf{G_2} =& h_{ux}\theta (\bs{E_{ux}S_{ux}}u) \bs{H_{uy}}(\bs{E_{ux}S_{ux}}u) \\
 &- (\bs{E_{ux}S_{ux}}u)^T \bs{H_{uy}} (\bs{E_{ux}}u - (E_{uv}\otimes I_{v2u})v) \\
 &+\frac{\tau}{2}(\bs{E_{ux}}u - (E_{uv}\otimes I_{v2u})v)^T \bs{H_{uy}} (\bs{E_{ux}}u - (E_{uv}\otimes I_{v2u})v), \\
  \mathbf{G_3} =& h_{vx}\theta (\bs{E_{vx}S_{vx}}v) \bs{H_{vy}} (\bs{E_{vx}S_{vx}}v) \\
 &- (\bs{E_{vx}S_{vx}}v)^T \bs{H_{vy}} ((E_{vu}\otimes I_{u2v})u - \bs{E_{vx}}v) \\
 &+\frac{\tau}{2} ((E_{vu}\otimes I_{u2v})u - \bs{E_{vx}}v)^T \bs{H_{vy}}  ((E_{vu}\otimes I_{u2v})u - \bs{E_{vx}}v).
\end{align*}
Clearly, $\mathbf{G_1}\geq 0$. By Young's inequality, we have $\mathbf{G_2}\geq 0$ and $\mathbf{G_3}\geq 0$ if $\tau\geq 1/(2\theta h_{ux})$ and $\tau\geq 1/(2\theta h_{vx})$, respectively. Therefore, the energy is conserved and the scheme is stable when \eqref{tau_limit} is satisfied.
\end{proof}
We note that in the scheme developed in \cite{Wang2016} the energy is greater or equal to $\mathbf{G}$ in \eqref{ddtG}, with the inequality resulted from the \emph{norm--contracting} condition.

\section{The wave equation on curvilinear grids}
In this section, we generalize the scheme to problems on curvilinear grids. We consider two cases: conforming blocks and non--conforming blocks, which are illustrated in Figure \ref{ConformingBlockCur} and \ref{NonconformingBlockCur}, respectively. 

\subsection{Numerical interface treatment of conforming blocks}\label{sec-conforming}
With only conforming blocks in the domain, the corners of adjacent blocks meet. We consider again the domain $\Omega=[0,1]^2$ but partitioned into two blocks $\Omega_u$ and $\Omega_v$ by a curved interface. The grids are then constructed independently in each block, see an illustration in Figure \ref{ConformingBlockCur}. The grids in each block are mapped to a Cartesian grid in a reference domain. The governing equations are also transformed from the physical domain to the reference domain, and the computation is performed in the reference domain. We refer to the textbook \cite{Knupp1993} for a detailed discussion on grid generation.   

\begin{figure}
     \centering
     \subfloat[]{\includegraphics[width=0.45\textwidth]{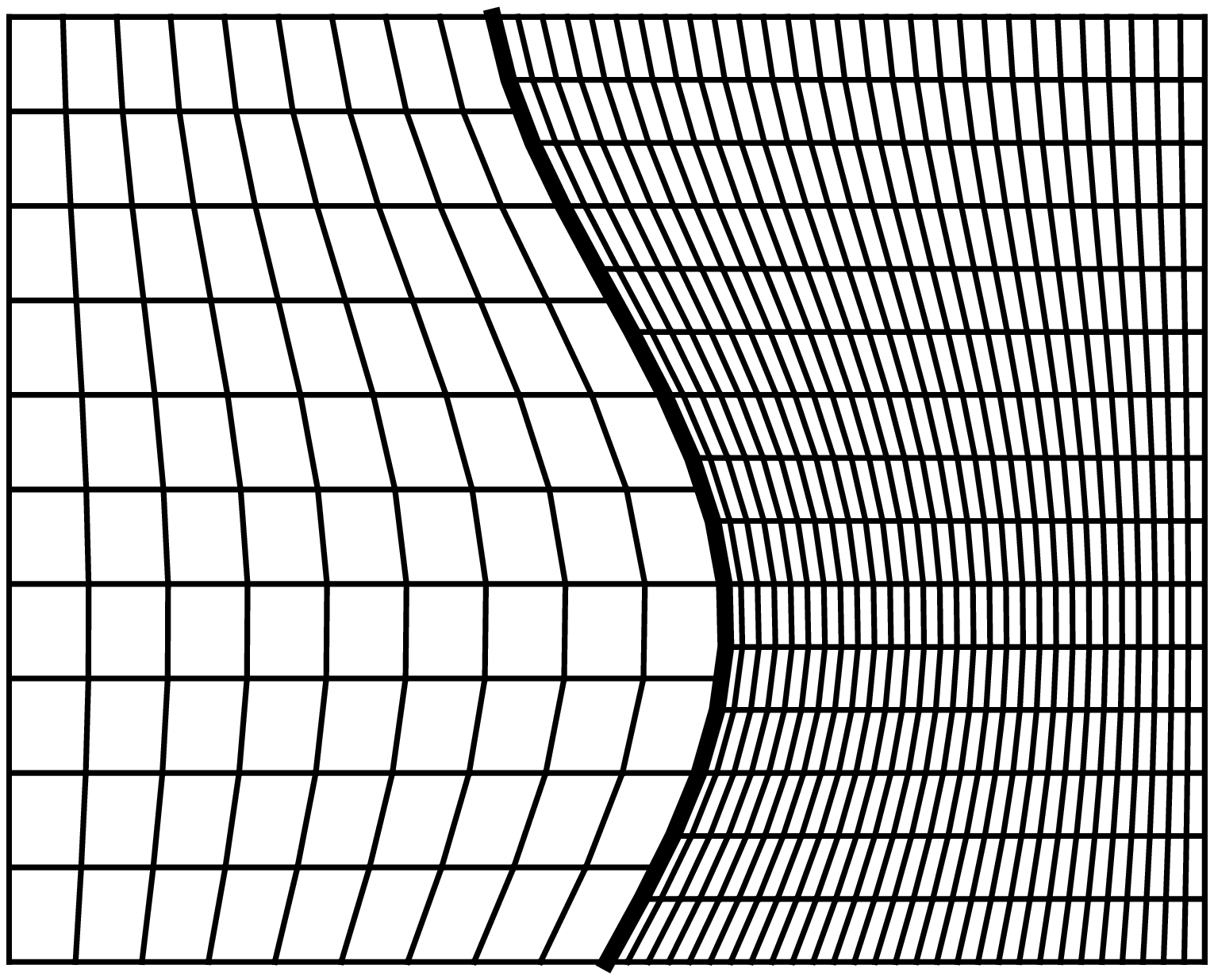}\label{ConformingBlockCur}}
     \subfloat[]{\includegraphics[width=0.45\textwidth]{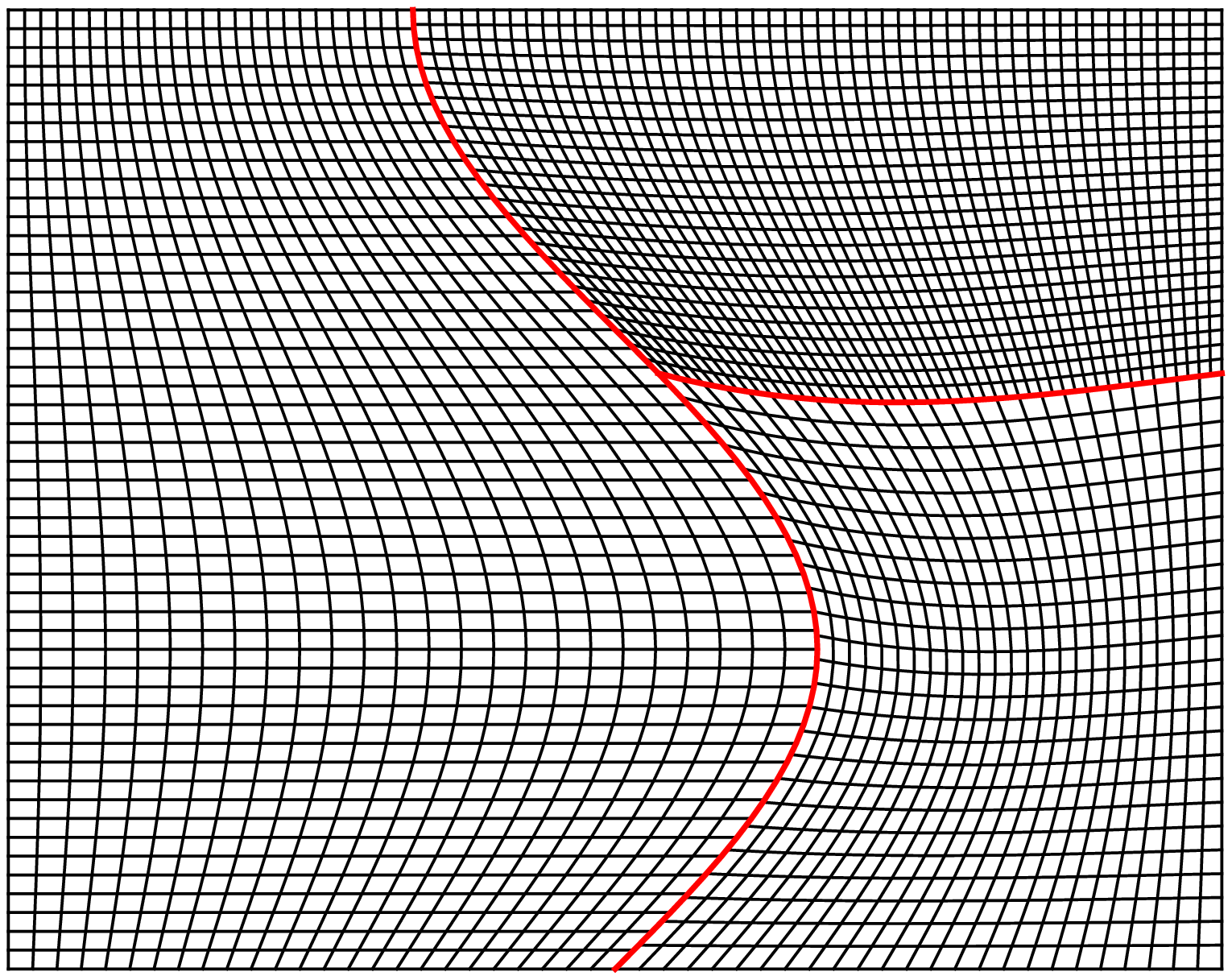}\label{NonconformingBlockCur}}
     \caption{Non--conforming interfaces with (a) conforming blocks (b) non--conforming blocks}
     \label{ConformingBlocks}
\end{figure}

The transformed equation in the reference domain can be derived by using the chain rule in calculus. We omit its derivation, and refer to  \cite{Almquist2014}. With $(x,y)$ denoting the coordinate in the reference domain, the unknown variables $U$ and $V$ are governed by the equation
\begin{equation}\label{UV}
\begin{split}
J_U U_{tt}&=(aU_x)_x+(cU_y)_y+(bU_y)_x+(bU_x)_y,\\
J_V V_{tt}&=(\alpha V_x)_x+(\gamma V_y)_y+(\beta V_y)_x+(\beta V_x)_y,
\end{split}
\end{equation} 
where the Jacobians $J_U,J_V>0$ and the variable coefficients satisfy $ac-b^2>0$ and $\alpha\gamma-\beta^2>0$. The interface conditions become 
\begin{equation}\label{UV_interface}
\begin{split}
&U=V,\\
&aU_x+bU_y=\alpha V_x+\beta V_y,
\end{split}
\end{equation} 
on $(x,y)\in\Gamma$. The coefficients in \eqref{UV} consist of metric derivatives that depend on the geometry of the physical domain and the transformation. The metric derivatives can either be computed analytically, or approximated to sufficient high accuracy. In the experiments in this paper, we choose the latter approach by using a tenth order finite difference stencil, making sure that the approximation of metric terms does not affect the overall accuracy of the numerical scheme.

As shown in \cite{Virta2014}, the equation \eqref{UV} together with the interface condition \eqref{UV_interface} admit a continuous energy estimate, thanks to the positive definiteness of the matrices $\begin{bmatrix} a & b \\ 
b & c
\end{bmatrix}$ and $\begin{bmatrix} \alpha & \beta \\ 
\beta & \gamma
\end{bmatrix}$ because of $ac-b^2>0$ and $\alpha\gamma-\beta^2>0$. 

The equations in \eqref{UV} are discretized by using the SBP operators, and the two blocks are patched together by the SAT method. The semi--discretized equations are
\begin{equation}\label{semi}
\begin{split}
&J_u u_{tt}=(\overline{D_{2ux}^{(a)}})u+(\overline{D_{2uy}^{(c)}})u+\boldsymbol{D_{1ux}}{\Lambda_b}\boldsymbol{D_{1uy}}u+\boldsymbol{D_{1uy}}{\Lambda_b}\boldsymbol{D_{1ux}}u+SAT_u, \\
&J_v v_{tt}=(\overline{D_{2vx}^{(\alpha)}})v+(\overline{D_{2vy}^{(\gamma)}})v+\boldsymbol{D_{1vx}}{\Lambda_\beta}\boldsymbol{D_{1vy}}v+\boldsymbol{D_{1vy}}{\Lambda_\beta}\boldsymbol{D_{1vx}}v+SAT_v,
\end{split}
\end{equation} 
where
\small
\begin{equation}\label{SATu}
\begin{split}
&SAT_u=\frac{1}{2}\boldsymbol{H_{ux}^{-1}}\boldsymbol{H_{uy}^{-1}}({\Lambda_a}\boldsymbol{E_{ux}S_{ux}}+{\Lambda_b}\boldsymbol{E_{ux}D_{1uy}})^T\boldsymbol{H_{uy}}(\boldsymbol{E_{ux}}u-(E_{uv}\otimes I_{v2u})v)\\
&-\frac{\tau}{2}\boldsymbol{H_{ux}^{-1}}((E_{ux}\otimes (I_{v2u}I_{u2v}))u-(E_{uv}\otimes I_{v2u})v)\\
&-\frac{\tau}{2}\boldsymbol{H_{ux}^{-1}}(\boldsymbol{E_{ux}}u-(E_{uv}\otimes I_{v2u})v)\\
&-\frac{1}{2}\boldsymbol{H_{ux}^{-1}}(({\Lambda_a}\boldsymbol{E_{ux}S_{ux}}+{\Lambda_b}\boldsymbol{E_{ux}D_{1uy}})u-(E_{uv}\otimes I_{v2u})({\Lambda_\alpha}\boldsymbol{E_{vx}S_{vx}}+{\Lambda_\beta}\boldsymbol{E_{vx}D_{1vy}})v),
\end{split}
\end{equation} 
\normalsize
and
\small
\begin{equation}\label{SATv}
\begin{split}
&SAT_v=-\frac{1}{2}\boldsymbol{H_{vx}^{-1}}\boldsymbol{H_{vy}^{-1}}({\Lambda_\alpha}\boldsymbol{E_{vx}S_{vx}}+{\Lambda_\beta}\boldsymbol{E_{vx}D_{1vy}})^T\boldsymbol{H_{vy}}(\boldsymbol{E_{vx}}v-(E_{vu}\otimes I_{u2v})u)\\
&-\frac{\tau}{2}\boldsymbol{H_{vx}^{-1}}((E_{vx}\otimes (I_{u2v}I_{v2u}))v-(E_{vu}\otimes I_{u2v})u)\\
&-\frac{\tau}{2}\boldsymbol{H_{vx}^{-1}}(\boldsymbol{E_{vx}}v-(E_{vu}\otimes I_{u2v})u)\\
&+\frac{1}{2}\boldsymbol{H_{vx}^{-1}}(({\Lambda_\alpha}\boldsymbol{E_{vx}S_{vx}}+{\Lambda_\beta}\boldsymbol{E_{vx}D_{1vy}})v-(E_{vu}\otimes I_{u2v})({\Lambda_a}\boldsymbol{E_{ux}S_{ux}}+{\Lambda_b}\boldsymbol{E_{ux}D_{1uy}})u).
\end{split}
\end{equation} 
\normalsize

We now clarify the notations in the semi--discretization \eqref{semi}. 
\begin{enumerate}
\item The mixed--derivative terms: $(bU_y)_x$ is approximated by $\boldsymbol{D_{1ux}}{\Lambda_b}\boldsymbol{D_{1uy}}u$, where ${\Lambda_b}$ is a diagonal matrix with diagonal entries $b(x,y)$ evaluated on the grid. The operators approximating the other three mixed--derivative terms are constructed in a similar way.
\item The variable--coefficient terms: In general the variable coefficients are functions of both $x$ and $y$, therefore an operator approximating $(aU_x)_x$ cannot be constructed by a single Kronecker product, but as a sum
\begin{equation*}
\overline{D_{2ux}^{(a)}}=\sum_{i=1}^{n_{uy}} D_{2ux}^{(a_i)}\otimes E_{uy}^i,
\end{equation*}
where $a_i$ is $a(x_i,y)$ evaluated on the grid, and $E_{uy}^i$ has value one in entry $(i,i)$ and zeros elsewhere. The second derivative operator $D_{2ux}^{(a_i)}$ is defined in Definition \ref{D2b} with the operator $M^{(a_i)}$ satisfying Lemma \ref{lemmaMb} and  \ref{lemmaMc}. The operators approximating the other three variable coefficient terms are constructed in a similar way.
\item The penalty terms: The two interpolation operators $I_{u2v}$ and $I_{v2u}$ constructed in \cite{Kozdon2016,Mattsson2010} satisfy the norm--compatible condition
\begin{equation*}
H_{uy}I_{v2u}=(H_{vy}I_{u2v})^T.
\end{equation*}
\end{enumerate}

Similar to the continuous case, the matrices $\Lambda_u=\begin{bmatrix} \Lambda_a & \Lambda_b \\ 
\Lambda_b & \Lambda_c
\end{bmatrix}$ and $\Lambda_v=\begin{bmatrix} \Lambda_\alpha & \Lambda_\beta \\ 
\Lambda_\beta & \Lambda_\gamma
\end{bmatrix}$ 
are positive definite. In fact, the  eigenvalues of $\Lambda_u$ and $\Lambda_v$ play an important role in the stability analysis. In particular, we will use the smallest eigenvalue
\begin{equation}\label{def_delta}
\delta = \frac{1}{2}\min\left(a_{ij}+c_{ij}-\sqrt{(a_{ij}-c_{ij})^2+4b^2_{ij}}, \alpha_{kl}+\gamma_{kl}-\sqrt{(\alpha_{kl}-\gamma_{kl})^2+4\beta^2_{kl}}\right), 
\end{equation}
where $i=1,2,\cdots, n_{ux},\ j=1,2,\cdots, n_{uy},\ k=1,2,\cdots, n_{vx},\ l=1,2,\cdots, n_{vy}$. Note that $\delta>0$, $a_{ij}-\delta\geq 0$ and $\alpha_{ij}-\delta\geq 0$. 


Let $a_{\max}$, $b_{\max}$, $\alpha_{\max}$ and $\beta_{\max}$ denote the maximum values of the variable coefficients $a(x,y)$, $b(x,y)$, $\alpha(x,y)$ and $\beta(x,y)$ evaluated on the interface, respectively. We also denote $h_{ux}$ and $h_{vx}$ the mesh size in the $x$--direction in $\Omega_u$ and $\Omega_v$, respectively. Stability of the semi--discretization is given by the following theorem.

\begin{theorem}\label{thm2}
The semi--discretization \eqref{semi}--\eqref{SATv} is stable if the interpolation operators $I_{u2v}$ and $I_{v2u}$ are norm--compatible and the penalty parameter $\tau$ satisfies 
\begin{equation*}
\tau\geq \max \left(\frac{a_{\max}^2+b_{\max}^2h_{ux}\sigma}{2h_{ux}\sigma\delta},\frac{\alpha_{\max}^2+\beta_{\max}^2h_{vx}\sigma}{2h_{vx}\sigma\delta}   \right),
\end{equation*}
where $\sigma$ is defined in Lemma \ref{lemmaMb}, and $\delta$ is defined in \eqref{def_delta}.
\end{theorem}
\begin{proof}
See Appendix.
\end{proof}

We remark that it is possible to use a penalty parameter that varies on grid points, and such a scheme is proposed in \cite{Virta2014} for problems with conforming interfaces. The penalty parameter $\tau$ used in Theorem \ref{thm2} corresponds to the largest value on all grid points. We also note that when the penalty parameter is chosen to be equal to the stability limit, accuracy deduction has been observed, and proved in some settings by a normal mode analysis \cite{Wang2017JSC}. Though there is no upper bound of $\tau$ for energy stability, a very large penalty parameter increases the spectral radius of the spatial discretization, and leads to a small time step. 

For accuracy, the SBP operators have truncation error $\mathcal{O}(h^{2p})$ in the interior and $\mathcal{O}(h^p)$ near the boundaries and the interface. The interpolation operators constructed in \cite{Kozdon2016,Mattsson2010} have truncation error $\mathcal{O}(h^{2p})$ in the interior of the interface, and truncation error $\mathcal{O}(h^p)$ on a few grid points near the edge of the interface. Therefore, in the semi--discretization \eqref{semi}, the largest truncation error is $\mathcal{O}(h^{p-2})$ introduced by the first three penalty terms in \eqref{SATu} and \eqref{SATv} because of $\tau,\boldsymbol{H_{ux}^{-1}},\boldsymbol{H_{uy}^{-1}}\sim\mathcal{O}(h^{-1})$. The truncation error  $\mathcal{O}(h^{p-2})$ is only localized on a few number of grid points at the corner of two adjacent blocks. According to the accuracy analysis in \cite{Wang2017}, we may expect a rate of convergence $p+1$ of the semi--discretization \eqref{semi}, the same as the scheme developed in \cite{Wang2016}. Note that the $p+1$ convergence rate is one order lower than the expected convergence rate when the grid interface is conforming. If the interpolation error at the edge could be improved to $\mathcal{O}(h^{p+1})$, then the expected rate of convergence would be $p+2$. However, it is proved in \cite{Lundquist} that such norm--compatible interpolation operators cannot be constructed. 

\subsection{Numerical interface treatment of non--conforming blocks}
An example of non--conforming blocks is shown in Figure \ref{NonconformingBlockCur}. The lower left corner of the upper right domain sits in the middle of the right boundary of the left domain. Such an interface configuration is sometimes called a T--junction interface.  

For a T--junction interface, the interface conditions must be imposed on the \emph{glue grid}, which is different from what is usually done for conforming blocks. The technique and the corresponding interpolation operators are constructed in \cite{Kozdon2016}, and are used for a T--junction interface on a Cartesian grid in \cite{Wang2016}. On a curvilinear grid, the discretization is performed in a similar way, but the grid transformation must be done carefully.

Coordinate transformation is performed block--wise, therefore an interface between two blocks is transformed twice. A common strategy is to transform each block in the physical domain to the unit square in the reference domain. This works well with conforming blocks. However, with nonconforming blocks such as in Figure \ref{NonconformingBlockCur} the interfaces are transformed differently in different blocks. As a consequence, the transformed equation must be scaled as explained in \cite{Kozdon2016} to obtain an energy stable scheme. The energy stable semi--discretization can then be constructed in a similar way as in the preceding sections.

\section{Numerical experiments}\label{sec_ex}
In this section, numerical experiments are performed to verify the stability and accuracy property of the numerical schemes developed in this paper. 
The diagonal norm SBP operators used in the numerical experiments can be found in \cite{Strand1994}
for $D_1\approx\partial/\partial x$ and in \cite{Mattsson2012} for $D_2^{(b)}\approx\partial/\partial x (b(x)\partial/\partial x)$. The interpolation operators can be found in \cite{Kozdon2016,Mattsson2010}. The L$_2$ errors are computed as the norm of the difference between the exact solution $u_{ex}$ and the numerical solution $u_h$ according to 
\begin{equation*}
\|u_{ex}-u_h\|_{\text{L}_2}=\sqrt{h_xh_y(u_{ex}-u_h)^T (u_{ex}-u_h)},
\end{equation*}
where $h_x$ and $h_y$ are the mesh size in the $x$ and $y$  spatial direction, respectively. 

\subsection{An extreme interface}
We consider the wave equation on the domain $[-1,1]\times[0,1]$, separated by the interface $x = 4\sin(7\pi y)/5$. The domain and mesh are depicted in Figure \ref{domain_d} and \ref{domain_d_mesh}, respectively. The aim of this experiment is to verify that the scheme is stable even when an interface with a large curvature is present in the domain, but not to test accuracy or convergence rate. As can be seen in Figure  \ref{domain_d_mesh}, the mesh is of bad quality due to large distortion.
\begin{figure}
     \centering
     \subfloat[]{\includegraphics[width=0.45\textwidth]{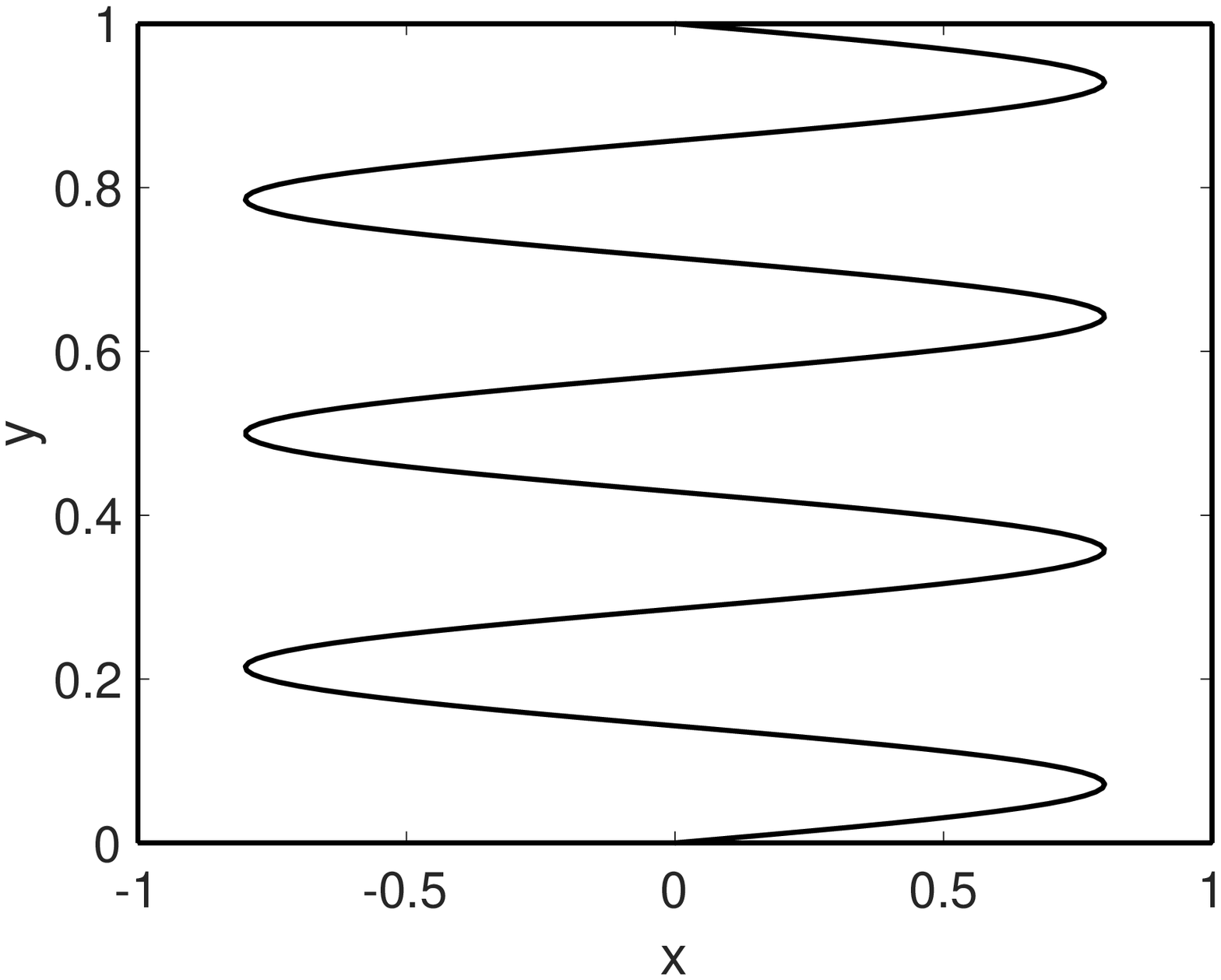}\label{domain_d}}
     \subfloat[]{\includegraphics[width=0.45\textwidth]{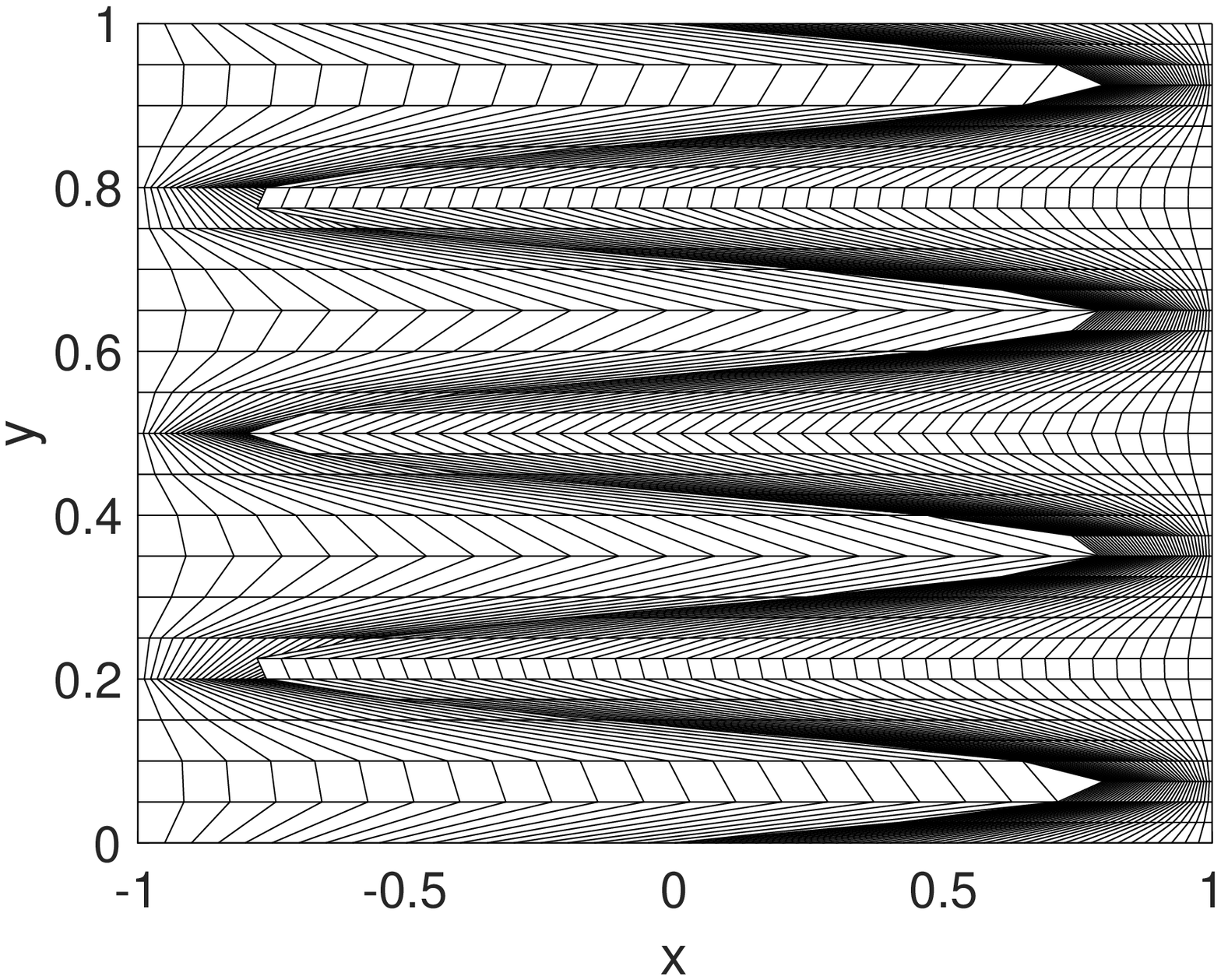}\label{domain_d_mesh}}
     \caption{(a) A composite domain with an extreme interface (b) Mesh}
\end{figure}

The wave equation is discretized by the sixth order SBP operators in each subdomain, and patched together by the SAT method using the sixth order interpolation operators \cite{Mattsson2010}. The semi--discretization can be written as a system of ordinary differential equations 
\begin{equation*}
w_{tt} = \boldsymbol{D} w + F,
\end{equation*}
where $w$ is the numerical solution, $\boldsymbol{D}$ is the spatial discretization operator including boundary and interface terms, and $F$ corresponds to the forcing function and boundary data evaluated on the grid. 

First, we use $21\times 21$ grid points in the left domain and $41\times 41$ grid points in the right domain, and perform an eigenvalue analysis. 
Stability requires that all the eigenvalues of $\boldsymbol{D}$ are real and non--positive. In Figure \ref{eigD}, we plot the eigenvalues of $\boldsymbol{D}$ multiplied by the square of the mesh size in the right domain, denoted by $\lambda$, and observe that they are indeed real and non--positive. 

Next, we test the scheme with a much finer mesh, $101\times 101$ grid points in the left domain and $201\times 201$ grid points in the right domain. Instead of an eigenvalue analysis, we perform a long time simulation by using the manufactured solution 
\begin{equation*}
U = \cos(x+1)\cos(y+2)\cos(\sqrt{2}t+3),
\end{equation*}
for initial and Neumann boundary data. We choose the classical Runge--Kutta method as the time integrator, and let the wave propagate for ten temporal periods. The L$_2$ error at each time step is plotted in Figure \ref{error_time_series}. We observe that the L$_2$ error is bounded in time.

\begin{figure}
     \centering
     \subfloat[]{\includegraphics[width=0.45\textwidth]{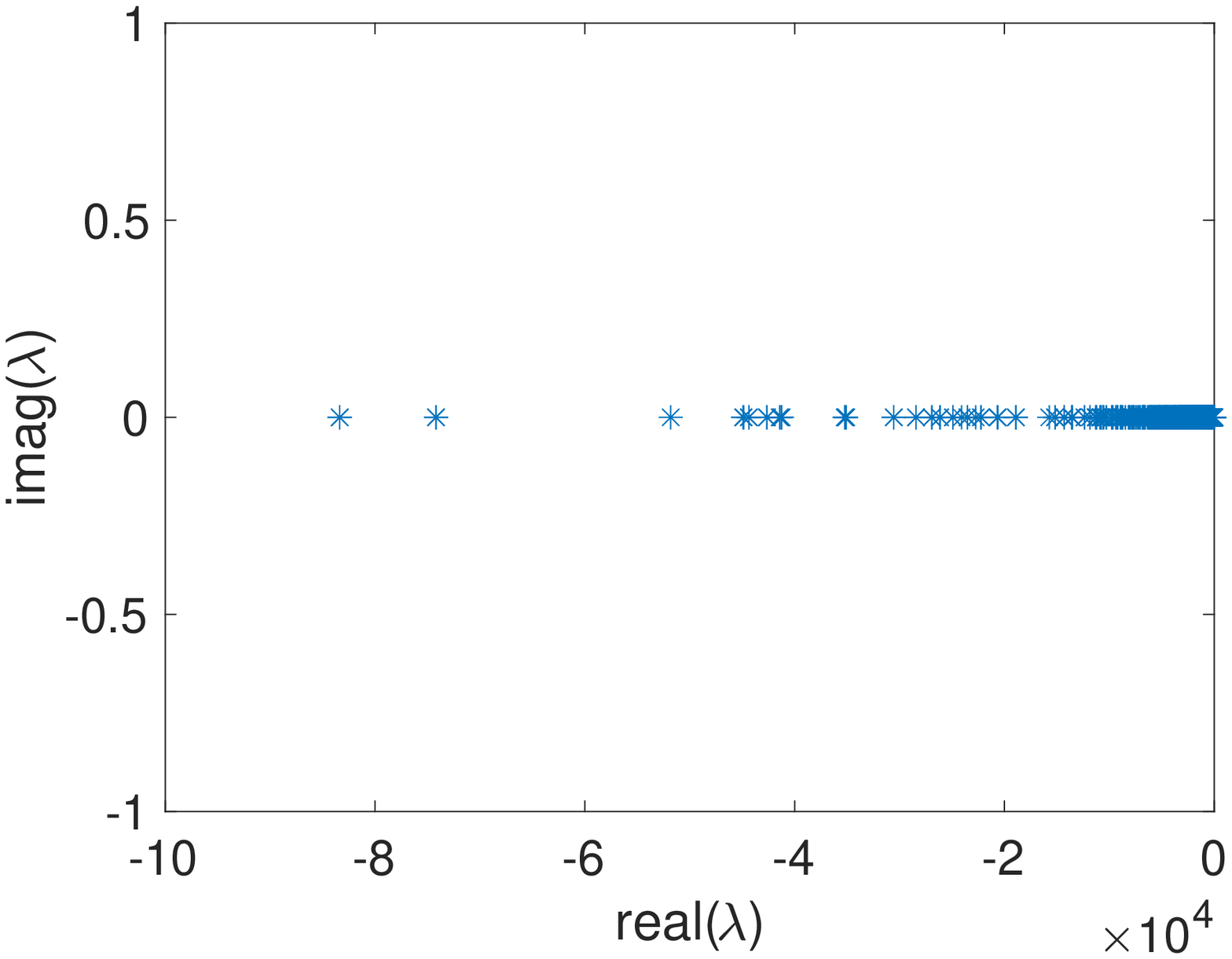}\label{eigD}}
     \subfloat[]{\includegraphics[width=0.45\textwidth]{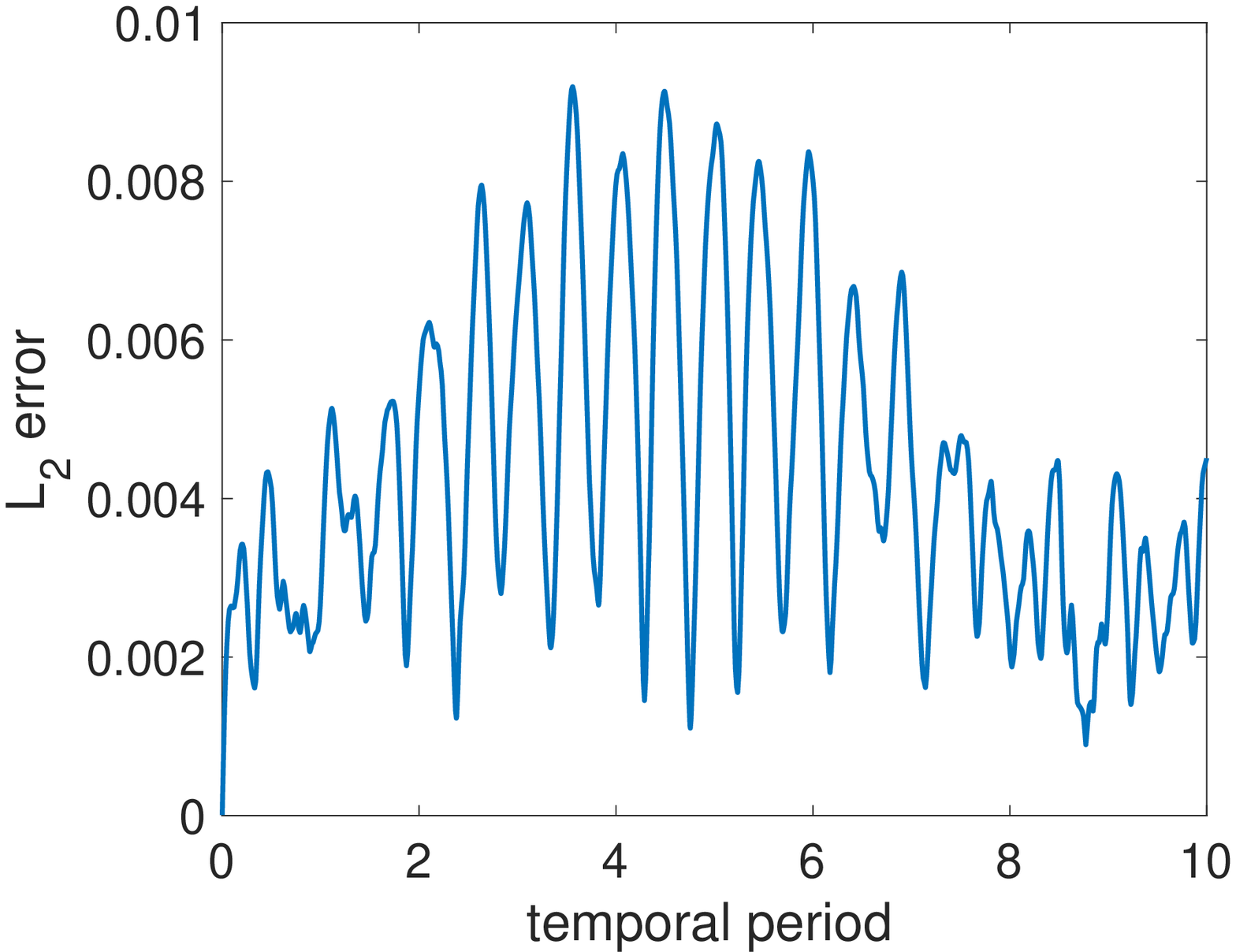}\label{error_time_series}}
     \caption{(a) Eigenvalues of the spatial discretization operator (b) L$_2$ error in ten temporal periods}
\end{figure}

\subsection{A T--junction interface}
We consider the same domain $[-1,1]\times [0,1]$ as in the previous experiment, but with interfaces depicted in Figure \ref{NonconformingBlockCur}. The interface in the vertical direction is defined by $x=\sin(3\pi y/2)/3$. The intersection point $(\bar{x},\bar{y})$ of the two interfaces is chosen by letting $\bar{y}=0.621$. The interface in the horizontal direction is defined by $y=\sin(\pi x/2)/5+\bar{y}-\sin(\pi \sin(3\pi\bar{y}/2)/6)/5$. The numbers of grid points in the left, lower right and upper right domain are $26\times 52, 26\times 26$ and $51\times 26$, respectively. Both the blocks and interfaces are non--conforming, see a close--up in Figure \ref{closeup}. When refining the mesh, the number of grid points is doubled in each spatial direction in each domain. 

To test accuracy and rate of convergence, we use the manufactured solution  
\begin{equation}\label{manufactured}
U = \cos(3\pi x+1) \cos(4\pi y+2) \cos(5\pi t+3).
\end{equation}
to obtain initial and Neumann boundary data, and propagate the wave until $t=2$. With this analytical solution, there is no forcing term in the equation. 

We solve the equation by the fourth and sixth order SBP--SAT finite difference method, and use the classical Runge--Kutta method to integrate in time. Since the interface in this experiment is a T--junction, we use the fourth and sixth order accurate interpolation operators constructed in \cite{Kozdon2016}. The time step is chosen small enough so that the error in the solution is determined by the spatial discretization. The errors in L$_2$ norm are shown in Figure \ref{result_t}, and the associated rates of convergence are given at the end of each error plot. In the figure, we use \emph{new SAT} as the legend to denote the results obtained by the scheme in this paper, and \emph{old SAT} to denote the result obtained by the scheme in \cite{Wang2016}.  The $x$--axis label $N$ is the number of grid points in the $x$--direction in the left domain.

We observe that the fourth and sixth order accurate scheme lead to third and fourth order convergence rate, respectively. This agrees well with the accuracy discussion in the end of Section \ref{sec-conforming} in this paper. We note that the sixth order method gives much smaller error than the fourth order methods. In addition, the four order method developed in this paper gives a smaller error than the fourth order method in \cite{Wang2016}. We have also performed an experiment with the sixth order method in \cite{Wang2016} and the numerical solution quickly blows up, indicating that method is unstable. This is not surprising because the energy analysis in \cite{Wang2016} requires the \emph{norm--compatible} condition of the interface operators, which are not satisfied by the sixth order operators.

\begin{figure}
     \centering
     \subfloat[]{\includegraphics[width=0.45\textwidth]{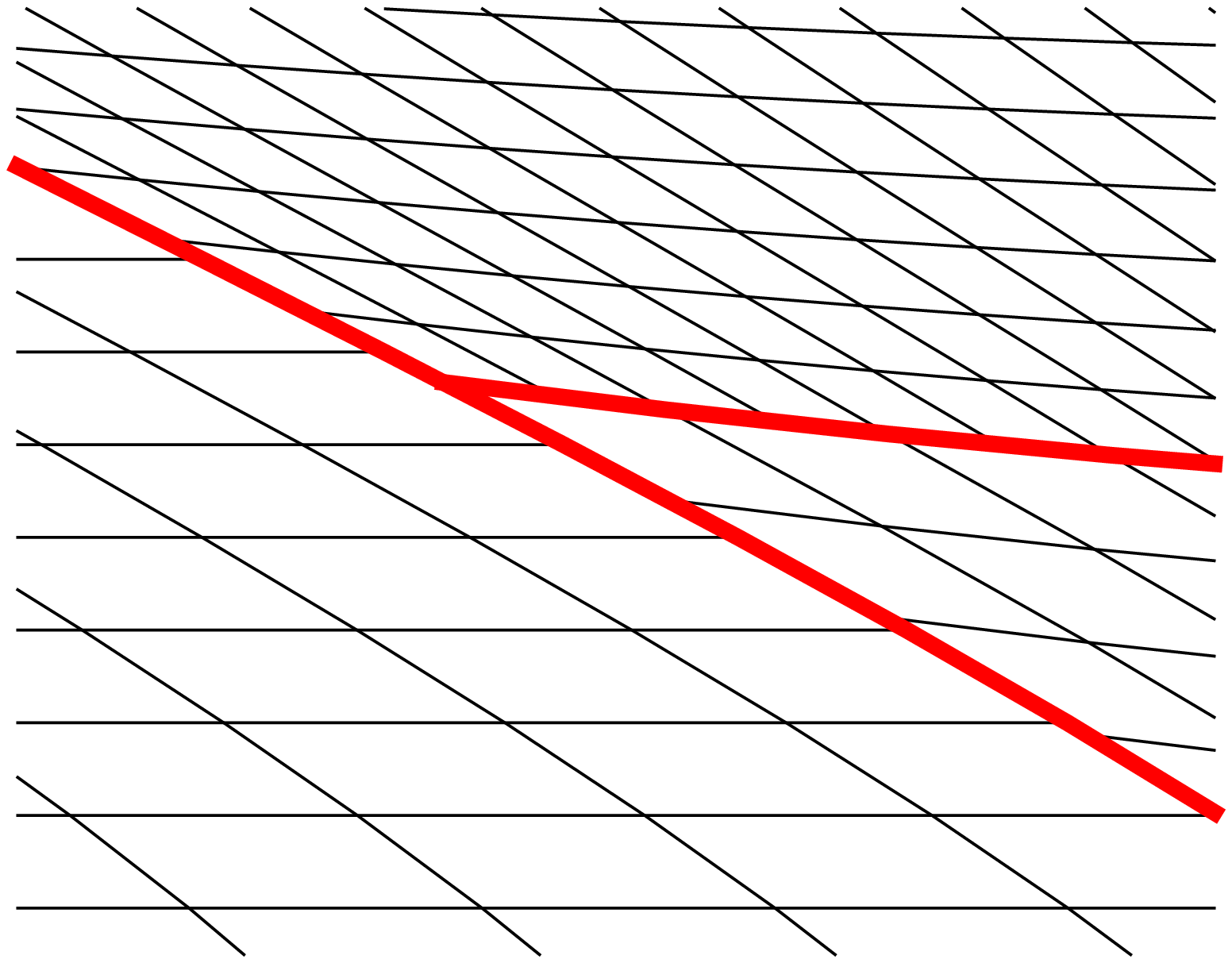}\label{closeup}}
     \subfloat[]{\includegraphics[width=0.45\textwidth]{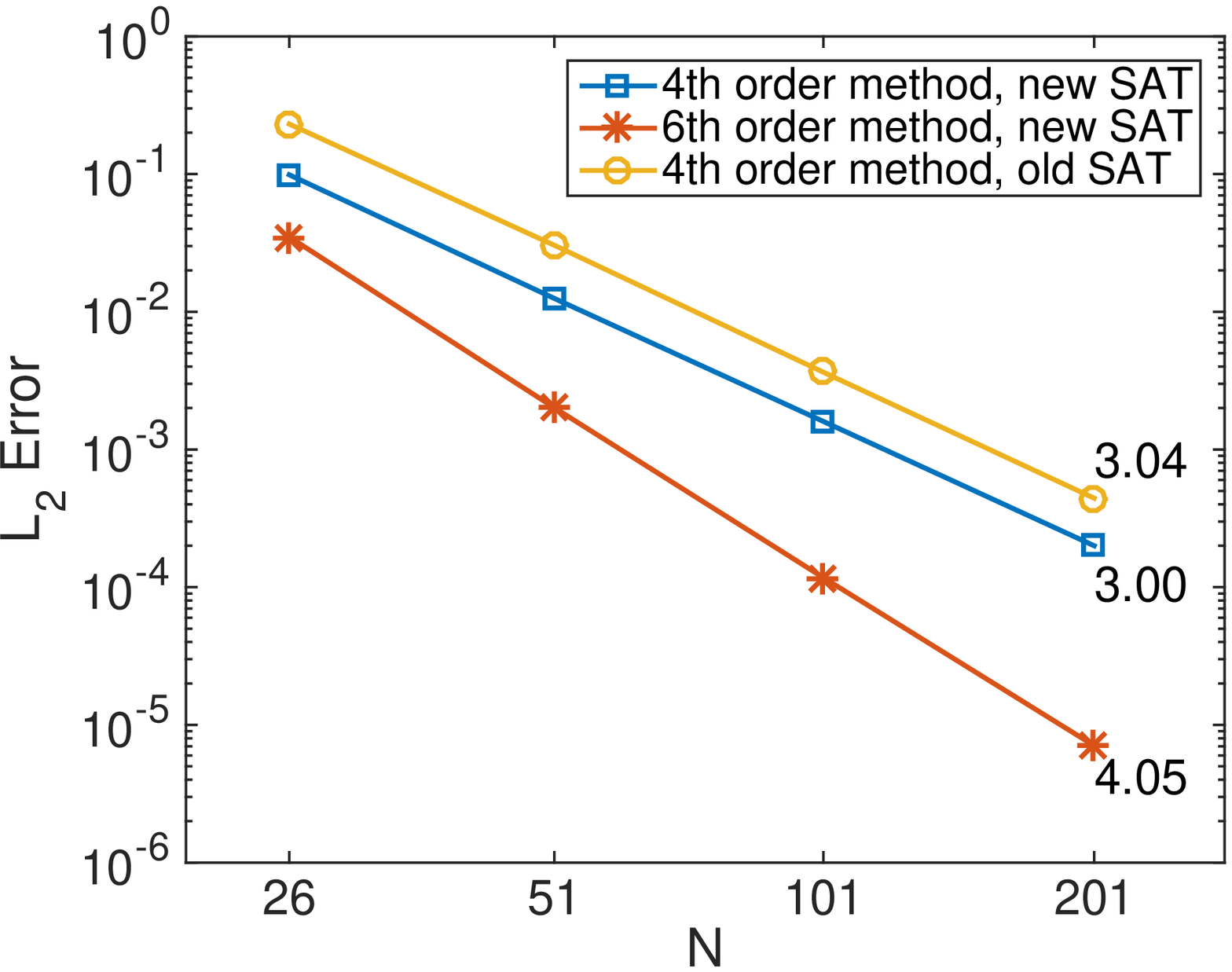}\label{result_t}}
     \caption{(a) A close--up of the interfaces (b) Rate of convergence}
     \label{ConformingBlocks}
\end{figure}

\section{Conclusion}
We use the SBP--SAT finite difference method to solve the wave equation on a composite domain. The domain is divided by curved interfaces resulting in non--conforming blocks, and the grid is constructed in each block independently resulting in non--conforming grid interfaces. We develop new penalty terms to patch the blocks together by the SAT method. This extends the provably stable scheme from fourth order accuracy \cite{Wang2016}  to sixth order accuracy. Numerical experiments demonstrate the superiority of the new sixth order accurate scheme. In addition, we find that the new fourth order accurate scheme is more accurate than the fourth order accurate scheme in \cite{Wang2016}.

We note that eighth order and tenth order interpolation operators are constructed in \cite{Kozdon2016}, and can potentially be incorporated into the developed scheme in this paper. However, higher than sixth order accurate SBP operators for second derivative with variable coefficient have not yet been constructed. 

\section*{Acknowledgement}
The author would like to thank Professor Gunilla Kreiss at Uppsala University for the fruitful discussion and support.

\section*{Appendix}
Stability is proved by the energy method, starting with multiplying the two semi--discretized equations in \eqref{semi} from the left by $u_t^T \boldsymbol{H_u}$ and $v_t^T \boldsymbol{H_v}$, respectively, where  $\boldsymbol{H_u}:=H_{ux}\otimes H_{uy}$ and $\boldsymbol{H_v}:=H_{vx}\otimes H_{vy}$. In the following derivation, we focus on the energy contribution from the first equation in \eqref{semi} as the energy contribution from the second equation in \eqref{semi}  can be computed in a similar way. 

First, we consider the energy contribution from the penalty term $SAT_u$:
\small
\begin{align}\label{p_SAT_u}
& u_t^T \boldsymbol{H_u} SAT_u \nonumber \\
= & \frac{1}{2} u_t^T ({\Lambda_a}\boldsymbol{E_{ux}S_{ux}}+{\Lambda_b}\boldsymbol{E_{ux}D_{1uy}})^T\boldsymbol{H_{uy}}(\boldsymbol{E_{ux}}u-(E_{uv}\otimes I_{v2u})v) \nonumber\\
&-\frac{\tau}{2} u_t^T \boldsymbol{H_{uy}}((E_{ux}\otimes (I_{v2u}I_{u2v}))u-(E_{uv}\otimes I_{v2u})v)\\
&-\frac{\tau}{2} u_t^T \boldsymbol{H_{uy}}(\boldsymbol{E_{ux}}u-(E_{uv}\otimes I_{v2u})v) \nonumber \\
&-\frac{1}{2} u_t^T \boldsymbol{H_{uy}}(({\Lambda_a}\boldsymbol{E_{ux}S_{ux}}+{\Lambda_b}\boldsymbol{E_{ux}D_{1uy}})u-(E_{uv}\otimes I_{v2u})({\Lambda_\alpha}\boldsymbol{E_{vx}S_{vx}}+{\Lambda_\beta}\boldsymbol{E_{vx}D_{1vy}})v). \nonumber
\end{align}
\normalsize
In the final energy estimate, we expect to see that the discrete energy is conserved in time.  We note that the first part of the third term in  \eqref{p_SAT_u} can be written as the time derivative of a quadratic term
\begin{equation*}
-\frac{\tau}{2} u_t^T \boldsymbol{H_{uy}} \boldsymbol{E_{ux}}u = -\frac{\tau}{4}\frac{d}{dt} ((\boldsymbol{E_{ux}}u)^T \boldsymbol{H_{uy}} \boldsymbol{E_{ux}}u).
\end{equation*}
With a positive $\tau$, the above term contributes positively to the discrete energy in terms of $\boldsymbol{E_{ux}}u$. By using the \emph{norm--compatible} property of the interpolation operators, we find that the first part of the second term in \eqref{p_SAT_u} can also be written as the time derivative of a quadratic term
\footnotesize
\begin{equation*}
-\frac{\tau}{2} u_t^T \boldsymbol{H_{uy}} (E_{ux}\otimes (I_{v2u}I_{u2v}))u = -\frac{\tau}{4}\frac{d}{dt} (((E_{ux}\otimes I_{u2v}) u)^T \boldsymbol{H_{vy}} (E_{ux}\otimes I_{u2v}) u),
\end{equation*}
\normalsize
which contributes to the energy positively in terms of $(E_{ux}\otimes I_{u2v})u$. We observe that in \eqref{p_SAT_u}, there are also terms $\boldsymbol{E_{ux}S_{ux}}u$ and $\boldsymbol{E_{ux}D_{1uy}}u$. Therefore, we need the corresponding positive energy contributions, which come from the SBP operators as shown below.

The energy contribution from the two mixed--derivative terms is
\begin{align}\label{mixed_e}
 &u_t^T \boldsymbol{H_u}\boldsymbol{D_{1ux}}\Lambda_b\boldsymbol{D_{1uy}}u+ u_t^T \boldsymbol{H_u}\boldsymbol{D_{1uy}}\Lambda_b\boldsymbol{D_{1ux}}u\\
 =& -u_t^T \boldsymbol{D_{1ux}^T} \boldsymbol{H_{u}}\Lambda_b\boldsymbol{D_{1uy}}u -u_t^T \boldsymbol{D_{1uy}^T}\boldsymbol{H_{u}} \Lambda_b\boldsymbol{D_{1ux}}u +u_t^T\boldsymbol{E_{ux}} \boldsymbol{H_{uy}}\Lambda_b \boldsymbol{D_{1uy}}u. \nonumber
\end{align}
Note that we have used the equality (\ref{D1_eq}) for $\boldsymbol{D_{1ux}}$ and $\boldsymbol{D_{1uy}}$, and have excluded boundary terms that do not correspond to the interface. On the right--hand side of \eqref{mixed_e}, the first two terms are volume terms, involving the numerical solution $u$ in the entire domain $\Omega_u$; the third term is an interface term, involving $u$ on the interface.

The energy contribution from the SBP approximation is
\begin{align}\label{sbp_e}
&u_t^T  \boldsymbol{H_u} (\overline{D_{2ux}^{(a)}}+\overline{D_{2uy}^{(c)}})u \nonumber\\
=& -u_t^T\boldsymbol{D_{1ux}^T} \boldsymbol{H_{u}}{\Lambda_a} \boldsymbol{D_{1ux}}u - u_t^T  \boldsymbol{H_{uy}}  \overline{R^{(a)}_{ux}}u + u_t^T \boldsymbol{H_{uy}} {\Lambda_a} \boldsymbol{E_{ux}S_{ux}}u \\
& -u_t^T \boldsymbol{D_{1uy}^T} \boldsymbol{H_{u}}{\Lambda_c}\boldsymbol{D_{1uy}}u - u_t^T  \boldsymbol{H_{ux}}  \overline{R^{(c)}_{uy}}u, \nonumber
\end{align} 
where the two remainder terms are $\overline{R^{(a)}_{ux}} = \sum_{i=1}^{n_{uy}} {R^{(a_i)}_{ux}}\otimes E^i_{uy}$ and $\overline{R^{(c)}_{uy}} = \sum_{i=1}^{n_{ux}} {R^{(c_i)}_{uy}}\otimes E^i_{ux}$. On the right--hand side of \eqref{sbp_e}, the third is an interface term, while the others are volume terms. For the volume terms in \eqref{mixed_e} and \eqref{sbp_e}, we have
\begin{align}
&-u_t^T \boldsymbol{D_{1ux}^T} \boldsymbol{H_{u}}\boldsymbol{\Lambda_b}\boldsymbol{D_{1uy}}u -u_t^T \boldsymbol{D_{1uy}^T}\boldsymbol{H_{u}} \boldsymbol{\Lambda_b}\boldsymbol{D_{1ux}}u -u_t^T\boldsymbol{D_{1ux}^T} \boldsymbol{H_{u}}\boldsymbol{\Lambda_a} \boldsymbol{D_{1ux}}u\nonumber \\
 & -u_t^T \boldsymbol{D_{1uy}^T} \boldsymbol{H_{u}}\boldsymbol{\Lambda_c}\boldsymbol{D_{1uy}}u - u_t^T  \boldsymbol{H_{uy}}  \overline{R^{(a)}_{ux}}u - u_t^T  \boldsymbol{H_{ux}}  \overline{R^{(c)}_{uy}} u \nonumber\\
 =&-\begin{bmatrix}
  \boldsymbol{D_{1ux}}u_t \\
  \boldsymbol{D_{1uy}}u_t
  \end{bmatrix}^T
  \begin{bmatrix}
  \boldsymbol{H_u} & \\
  & \boldsymbol{H_u}
  \end{bmatrix}
    \begin{bmatrix}
 {\Lambda_a} & { \Lambda_b} \\
   {\Lambda_b} &  {\Lambda_c}
  \end{bmatrix}
  \begin{bmatrix}
  \boldsymbol{D_{1ux}}u \\
  \boldsymbol{D_{1uy}}u
  \end{bmatrix}-u_t^T  \boldsymbol{H_{uy}}  \overline{R^{(a)}_{ux}} u - u_t^T  \boldsymbol{H_{ux}}  \overline{R^{(c)}_{uy}} u.\nonumber
\end{align} 

Next, we split $\begin{bmatrix}
{\Lambda_a} & { \Lambda_b} \\
   {\Lambda_b} &  {\Lambda_c}
  \end{bmatrix}$ into two parts, and obtain

\begin{align}
&-u_t^T \boldsymbol{D_{1ux}^T} \boldsymbol{H_{u}}\boldsymbol{\Lambda_b}\boldsymbol{D_{1uy}}u -u_t^T \boldsymbol{D_{1uy}^T}\boldsymbol{H_{u}} \boldsymbol{\Lambda_b}\boldsymbol{D_{1ux}}u -u_t^T\boldsymbol{D_{1ux}^T} \boldsymbol{H_{u}}\boldsymbol{\Lambda_a} \boldsymbol{D_{1ux}}u\nonumber \\
 & -u_t^T \boldsymbol{D_{1uy}^T} \boldsymbol{H_{u}}\boldsymbol{\Lambda_c}\boldsymbol{D_{1uy}}u - u_t^T  \boldsymbol{H_{uy}}  \overline{R^{(a)}_{ux}}u - u_t^T  \boldsymbol{H_{ux}}  \overline{R^{(c)}_{uy}} u \nonumber\\
  =&-\begin{bmatrix}
  \boldsymbol{D_{1ux}}u_t \\
  \boldsymbol{D_{1uy}}u_t
  \end{bmatrix}^T
  \begin{bmatrix}
  \boldsymbol{H_u} & \\
  & \boldsymbol{H_u}
  \end{bmatrix}
    \left(\begin{bmatrix}
  {\Lambda_a} & { \Lambda_b} \\
   {\Lambda_b} &  {\Lambda_c}
  \end{bmatrix}-\delta \boldsymbol{I}\right)
  \begin{bmatrix}
  \boldsymbol{D_{1ux}}u \\
  \boldsymbol{D_{1uy}}u
  \end{bmatrix} \label{abc1}\\ 
  &-\delta\begin{bmatrix}
  \boldsymbol{D_{1ux}}u_t \\
  \boldsymbol{D_{1uy}}u_t
  \end{bmatrix}^T
  \begin{bmatrix}
  \boldsymbol{H_u} & \\
  & \boldsymbol{H_u}
  \end{bmatrix}
  \begin{bmatrix}
  \boldsymbol{D_{1ux}}u \\
  \boldsymbol{D_{1uy}}u
  \end{bmatrix}-u_t^T  \boldsymbol{H_{uy}}  \overline{R^{(a)}_{ux}} u - u_t^T  \boldsymbol{H_{ux}}  \overline{R^{(c)}_{uy}} u,\nonumber
\end{align} 
where $\boldsymbol{I}$ is an identity operator, and $\delta$ is defined in \eqref{def_delta}. The first term in \eqref{abc1} is the time derivative of a non--positive quantity.  From the second term in \eqref{abc1}, we need to get quadratic terms for both $\boldsymbol{E_{ux}D_{1ux}}u$ and $\boldsymbol{E_{ux}S_{ux}}u$ on the interface. We write
\begin{align}
&-\delta\begin{bmatrix}
  \boldsymbol{D_{1ux}}u_t \\
  \boldsymbol{D_{1uy}}u_t
  \end{bmatrix}^T
  \begin{bmatrix}
  \boldsymbol{H_u} & \\
  & \boldsymbol{H_u}
  \end{bmatrix}
  \begin{bmatrix}
  \boldsymbol{D_{1ux}}u \\
  \boldsymbol{D_{1uy}}u
  \end{bmatrix}-u_t^T  \boldsymbol{H_{uy}}  \overline{R^{(a)}_{ux}} u - u_t^T  \boldsymbol{H_{ux}}   \overline{R^{(c)}_{uy}} u\nonumber\\
  =& -\delta (\boldsymbol{D_{1ux}}u_t)^T \boldsymbol{H_u}  (\boldsymbol{D_{1ux}}u) -u_t^T  \boldsymbol{H_{uy}}  \overline{R^{(a)}_{ux}} u \nonumber \\
  &-\delta (\boldsymbol{D_{1uy}}u_t)^T \boldsymbol{H_u}  (\boldsymbol{D_{1uy}}u) -u_t^T  \boldsymbol{H_{ux}}  \overline{R^{(c)}_{uy}} u\nonumber \\
  =&-u_t^T (( \delta D_{1ux}^T H_{ux} D_{1ux}+R^{(\delta)}_{ux} )\otimes H_{uy} )u - u_t^T  \boldsymbol{H_{uy}}  \overline{R^{(a-\delta)}_{ux}}  u\nonumber\\
  &-\delta (\boldsymbol{E_{ux}D_{1uy}}u_t)^T \boldsymbol{H_u}  (\boldsymbol{E_{ux}D_{1uy}}u) -u_t^T  \boldsymbol{H_{ux}}  \overline{R^{(c)}_{uy}}u\nonumber\\
  &-\delta ((\boldsymbol{I_{ux}} -\boldsymbol{E_{ux}} ) \boldsymbol{D_{1uy}}u_t)^T \boldsymbol{H_u}  (\boldsymbol{I_{ux}} -\boldsymbol{E_{ux}} )\boldsymbol{D_{1uy}}u \nonumber\\
  =&-u_t^T (M^{(\delta)}\otimes H_{uy}) u - u_t^T  \boldsymbol{H_{uy}}  \overline{R^{(a-\delta)}_{ux}}  u \nonumber\\
  &-\delta (\boldsymbol{E_{ux}D_{1uy}}u_t)^T \boldsymbol{H_u}  (\boldsymbol{E_{ux}D_{1uy}}u) -u_t^T  \boldsymbol{H_{ux}}  \overline{R^{(c)}_{uy}}u\label{nct}\\
  &-\delta ((\boldsymbol{I_{ux}} -\boldsymbol{E_{ux}} ) \boldsymbol{D_{1uy}}u_t)^T \boldsymbol{H_u}  (\boldsymbol{I_{ux}} -\boldsymbol{E_{ux}} )\boldsymbol{D_{1uy}}u \nonumber \\
    =&-u_t^T (\widetilde M^{(\delta)}\otimes H_{uy}) u - h_{ux}\sigma\delta (\boldsymbol{E_{ux}S_{ux}}u_t)^T\boldsymbol{H_{uy}E_{ux}S_{ux}}u - u_t^T  \boldsymbol{H_{uy}}   \overline{R^{(a-\delta)}_{ux}}  u\nonumber \\
  &-\delta (\boldsymbol{E_{ux}D_{1uy}}u_t)^T \boldsymbol{H_{uy}}  (\boldsymbol{E_{ux}D_{1uy}}u) -u_t^T  \boldsymbol{H_{ux}}  \overline{R^{(c)}_{uy}}u\label{bt}\\
  &-\delta ((\boldsymbol{I_{ux}} -\boldsymbol{E_{ux}} ) \boldsymbol{D_{1uy}}u_t)^T \boldsymbol{H_u}  (\boldsymbol{I_{ux}} -\boldsymbol{E_{ux}} )\boldsymbol{D_{1uy}}u.\nonumber 
\end{align} 
Note that in the above derivation, we use Lemma \ref{lemmaMc} to obtain \eqref{nct}, and Lemma \ref{lemmaMb} to obtain \eqref{bt}. We have obtained both the time derivative of quadratic terms for $\boldsymbol{E_{ux}S_{ux}}u$ and $\boldsymbol{E_{ux}D_{1ux}}u$. 

After a very similar derivation of energy contribution for the second equation in \eqref{semi}, we move all terms to one side and write it in the form $\frac{d}{dt}\boldsymbol{G}=0$. The final step of the energy analysis is to determine the penalty parameter $\tau$ so that $\boldsymbol{G}$ is a discrete energy satisfying $\boldsymbol{G}\geq 0$. With some algebraic calculations, we may write the energy contribution from all interface terms as $x_I^T \boldsymbol{A}x_I$, where $x_I$ is the vector in the form
\begin{align*}
x_I=\begin{bmatrix}\boldsymbol{E_{ux}}u; \boldsymbol{E_{ux}S_{ux}}u; \boldsymbol{E_{ux}D_{1ux}}u; (E_{uv}\otimes I_{v2u})v \\
\boldsymbol{E_{vx}}v; \boldsymbol{E_{vx}S_{vx}}v; \boldsymbol{E_{vx}D_{1vx}}v; (E_{vu}\otimes I_{u2v})u \end{bmatrix}.
\end{align*} 
The matrix $\boldsymbol{A}$ is a block matrix in the form $\boldsymbol{A}=\begin{bmatrix} \boldsymbol{A_1}, &\boldsymbol{0}, \\
\boldsymbol{0}, & \boldsymbol{A_2}
\end{bmatrix}$, where
\begin{align*}
\boldsymbol{A_1} =-\begin{bmatrix}
-\frac{\tau}{4}\boldsymbol{H_{uy}}, & \frac{1}{4}\boldsymbol{H_{uy}}\Lambda_{a}, &  \frac{1}{4}\boldsymbol{H_{uy}}\Lambda_{b}, &  \frac{\tau}{4}\boldsymbol{H_{uy}}\\
 \frac{1}{4}\boldsymbol{H_{uy}}\Lambda_{a}, &-\frac{h_{ux}\sigma\delta}{2}\boldsymbol{H_{uy}}, & 0,& -\frac{1}{4}\boldsymbol{H_{uy}}\Lambda_{a} \\
  \frac{1}{4}\boldsymbol{H_{uy}}\Lambda_{b}, &0,  & -\frac{\delta}{2}\boldsymbol{H_{uy}}, & -\frac{1}{4}\boldsymbol{H_{uy}}\Lambda_{b}\\
  \frac{\tau}{4}\boldsymbol{H_{uy}}, & -\frac{1}{4}\boldsymbol{H_{uy}}\Lambda_{a}, & -\frac{1}{4}\boldsymbol{H_{uy}}\Lambda_{b}, & -\frac{\tau}{4}\boldsymbol{H_{uy}}
  \end{bmatrix},
  \end{align*} 
\begin{align*}
\boldsymbol{A_2} =-\begin{bmatrix}
-\frac{\tau}{4}\boldsymbol{H_{vy}}, & -\frac{1}{4}\boldsymbol{H_{vy}}\Lambda_{\alpha}, &  -\frac{1}{4}\boldsymbol{H_{vy}}\Lambda_{\beta}, &  \frac{\tau}{4}\boldsymbol{H_{vy}}\\
 -\frac{1}{4}\boldsymbol{H_{vy}}\Lambda_{\alpha}, &-\frac{h_{vx}\sigma\delta}{2}\boldsymbol{H_{vy}}, & 0,&  \frac{1}{4}\boldsymbol{H_{vy}}\Lambda_{\alpha}\\
  -\frac{1}{4}\boldsymbol{H_{vy}}\Lambda_{\beta}, &0,  & -\frac{\delta}{2}\boldsymbol{H_{vy}}, & \frac{1}{4}\boldsymbol{H_{vy}}\Lambda_{\beta} \\
\frac{\tau}{4}\boldsymbol{H_{vy}}, & \frac{1}{4}\boldsymbol{H_{vy}}\Lambda_{\alpha}, & \frac{1}{4}\boldsymbol{H_{vy}}\Lambda_{\beta}, & -\frac{\tau}{4}\boldsymbol{H_{vy}}
\end{bmatrix}.
\end{align*} 
We note that $\boldsymbol{A_1}$ is symmetric, and each submatrix of $\boldsymbol{A_1}$ is a diagonal matrix of dimension $n_{uy}$. Therefore, we can write $\boldsymbol{A_1}$ as a sum of $n_{uy}$ matrices, where the $i^{th}$ matrix takes the $i^{th}$ diagonal element of each submatrix with all the other elements zero. We then compute the eigenvalues of each matrix, and require them to be non--negative. By considering $\boldsymbol{A_2}$ in the same way, we obtain the limit on the penalty parameter for which $\boldsymbol{A}$ is positive semi--definite
\begin{equation*}
\tau\geq \max \left(\frac{a_{\max}^2+b_{\max}^2h_{ux}\sigma}{2h_{ux}\sigma\delta},\frac{\alpha_{\max}^2+\beta_{\max}^2h_{vx}\sigma}{2h_{vx}\sigma\delta}   \right),
\end{equation*}
where $a_{\max}$, $b_{\max}$, $\alpha_{\max}$ and $\beta_{\max}$ are the maximum of functions $a(x,y)$, $b(x,y)$, $\alpha(x,y)$ and $\beta(x,y)$ evaluated on the interface, respectively.

\end{document}